\documentclass[preprint,12pt]{elsarticle}
\usepackage{amsthm,amsfonts,amssymb,amscd,amsmath,enumerate,verbatim,calc,graphicx,geometry}
\usepackage[all]{xy}
\newtheorem{theorem}{Theorem}[section]
\newtheorem{lemma}[theorem]{Lemma}
\newtheorem{proposition}[theorem]{Proposition}
\newtheorem{corollary}[theorem]{Corollary}
\theoremstyle{definition}
\theoremstyle{definitions}
\newtheorem{definition}[theorem]{Definition}

\newtheorem{remark}[theorem]{Remark}
\newtheorem{example}[theorem]{Example}
\theoremstyle{notations}

\theoremstyle{remarks}

\newcommand{\N}{\mathbb{N}}

\newcommand{\sub}{\subseteq}

\newcommand{\ov}{\overline}

\newcommand{\lo}{\longrightarrow}

\newcommand{\wt}{\widetilde}
\newcommand{\vf}{\varphi}
\newcommand{\fr}{\frac}
\newcommand{\al}{\alpha}
\newcommand{\la}{\lambda}

\newcommand{\bt}{\beta}
\newcommand{\ti}{\tilde}
\newcommand{\tx}{\textit}

\newcommand{\psg}{\pi_1^{sg}(X,x)}
\newcommand{\psp}{\pi_1^{sp}(X,x)}
\newcommand{\pbs}{\pi_1^{bsp}(X,x)}
\newcommand{\pt}{\pi_1^{qtop}(X,x)}
\newcommand{\pc}{p:\wt{X}\lo X}
\newcommand{\pst}{p_*\pi_1(\wt{X},\ti{x})}
\newcommand{\V}{\mathcal{V}}
\newcommand{\U}{\mathcal{U}}
\journal{ }

\begin{document}

\begin{frontmatter}



\title{Spanier spaces and covering theory of non-homotopically path Hausdorff spaces}


\author[]{Ali~Pakdaman}
\ead{Alipaky@yahoo.com}
\author[]{Hamid~Torabi}
\ead{hamid$_{-}$torabi86@yahoo.com}
\author[]{Behrooz~Mashayekhy\corref{cor1}}
\ead{bmashf@um.ac.ir}
\address{Department of Pure Mathematics, Center of Excellence in Analysis on Algebraic Structures, Ferdowsi University of Mashhad,\\
P.O.Box 1159-91775, Mashhad, Iran.}
\cortext[cor1]{Corresponding author}
\begin{abstract}
H. Fischer et al. (Topology and its Application, 158 (2011) 397-408.) introduced the Spanier group of a based space $(X,x)$ which is denoted by $\psp$.
By a Spanier space we mean a space $X$ such that $\psp=\pi_1(X,x)$, for every $x\in X$. In this paper, first we give an example of Spanier spaces. Then we study the influence of the Spanier group on covering theory and introduce Spanier coverings which are universal coverings in the categorical sense. Second, we give a necessary and sufficient condition for the existence of Spanier
coverings for non-homotopically path Hausdorff spaces. Finally, we study the topological properties of Spanier groups and find out a criteria for the Hausdorffness of topological fundamental groups.

\end{abstract}


\end{frontmatter}


\section{Introduction and motivation}
 A continuous map $p:\wt{X}\lo X$ is a $\textit {covering}$ of $X$, and $\wt{X}$ is called a $\textit {covering space}$ of $X$, if for every $x\in X$ there exists an open
subset $U$ of $X$ with $x\in U$ such that $U$ is $\tx{evenly covered}$ by $p$, that is,
$p^{-1}(U)$ is a disjoint union of open subsets of $\wt{X}$ each of which is mapped
homeomorphically onto $U$ by $p$.
For a connected, locally path connected and semi-locally simply connected space $X$, it is well known that there
is a 1-1 correspondence between its connected covering spaces and open subgroups of
its fundamental group $\pi_1(X,x)$, for a point $x\in X$ \cite{S}. E. H. Spanier \cite{S} classified connected covering spaces of the space $X$ using some subgroups of the fundamental group of $X$, recently named Spanier groups. If $\U$ is an open cover of $X$, the subgroup of $\pi_1(X, x)$ consisting of the homotopy classes of loops that can be represented by a
product of the following type: $$\prod\limits_{j=1}^{n}u_jv_ju_j^{-1},$$
where the $u_j$'s are arbitrary paths starting at the base point $x$ and each $v_j$ is a loop inside one of the neighborhoods $U_j\in\mathcal{U}$.
This group is called the unbased \emph{Spanier group with respect to $\U$}, denoted by $\pi(\U,x)$ \cite{S, R}. The following theorem is an interesting result on the above notion.

\begin{theorem}(\cite{S})
For a connected, locally path connected space $X$, if $H$ is a subgroup of $\pi_1(X,x)$ for $x\in X$ and there exists an open cover $\U$ of $X$ such that $\pi(\U,x)\leq H$, then there exists a covering $\pc$ such that $p_*\pi_1(\wt{X},\ti{x})=H$.
\end{theorem}
Since for a locally path connected and semi-locally simply connected space $X$ there exists an open cover $\U$ such that $\pi(\U,x)=1$, for a point $x\in X$, the existence of simply connected universal covering follows from the above theorem.
But without locally path connectedness, these results fail since there exists a semi-locally simply connected space with nontrivial Spanier groups corresponding to every its open cover. H. Fischer, D. Repovs, Z.Virk, A. Zastrow \cite{R} proposed a modification of Spanier groups so that the corresponding results will be correct for all spaces. In order to do this, they instead of open sets U also considered ``pointed open sets", i.e. pairs $(U, x)$, where $x\in U$ and $U$ is open.
Let $\mathcal{U}=\{U_i | i\in I\}$ be a cover of X by open sets. For each $U_i\in\mathcal{U}$ take $|U_i|$ copies into $\V$ and define each of those copies as $(U_i,p)$, i.e. use the same set $U_i$ as first entry, and let the second entry run over all points
$p\in U_i$.
\begin{definition}(\cite{R}) Let $X$ be a space, $x\in X$, and $\V=\{(V_i,x_i)| i\in I\}$ be a cover of $X$ by open neighborhood pairs. Then
let $\pi^*(\V, x)$ be the subgroup of $\pi_1(X, x)$ which contains all homotopy classes having representatives of the following
type:$$\prod\limits_{j=1}^{n}u_jv_ju_j^{-1},$$
where the $u_j$'s are arbitrary paths that run from $x$ to some point $x_i$ and each $v_j$ then must be a closed path inside the
corresponding $V_i$. This group is called the based Spanier group with respect to $\mathcal{V}$, denoted by $\pi^*(\V,x)$.
\end{definition}
Let $\mathcal{U}$ be a refinement of an open covering $\mathcal{V}$. Then $\pi(\U)\leq\pi(\V)$ and $\pi^*(\U)\leq\pi^*(\V)$ if $\mathcal{U}$ and $\mathcal{V}$ are open coverings of pointed sets. By these inclusions, there exist inverse limits of these Spanier groups, defined via the directed system of all coverings with respect to refinement. H. Fisher et al. \cite{R} called them the unbased Spanier group and the based Spanier group of the space $X$ which we denote them by $\psp$ and $\pbs$, respectively. They also mentioned that these inverse limits are realized by intersections as follows:
 $$\psp=\bigcap\limits_{open\ covers\ \U}\pi(\U,x),$$ $$\pbs=\bigcap\limits_{open\ covers\ \V\ by\ pointed\ sets}\pi^*(\V,x).$$

For the spaces that are not locally nice, classification of covering spaces is not as pleasant. H. Fischer and A. Zastrow in \cite{FZ} defined a generalized regular covering which enjoys
most of the usual properties of classical coverings, with the possible exception of
evenly covered neighborhoodness. If $X$ is connected, locally path-connected and semi-locally simply connected,
then the generalized universal covering $p:\wt{X}\lo X$ agrees with the classical universal covering.
While semi-local simple connectivity is a crucial condition in classical covering space theory, the generalized covering
space theory mainly considered the condition called ``homotopically Hausdorf''. A space $X$ is $\textit{homotopically Hausdorff}$ if given any point $x$ in $X$ and any nontrivial homotopy class
$[\al]\in\pi_1(X,x)$, then there is a neighborhood $U$ of $x$ which contains no representative for $[\al]$.

H. Fischer et al. \cite{R} proved that triviality of based Spanier group implies the existence of generalized universal covering. In fact, they proved that for a space $X$, if $\pbs=1$, then $X$ is homotopically path Hausdorff which implies the existence of generalized universal covering.

Although all homotopically path Hausdorff spaces are homotopically Hausdorff, but there exist non-homotopically path Hausdorff spaces which are homotopically Hausdorff \cite[Proposition 3.4]{R}. Also, the authors \cite{P2, T2} studied the covering theory of non-homotopically Hausdorff spaces. Accordingly, we would like to study coverings
of non-homotopically path Hausdorff spaces and investigate the topology type of their fundamental group and their
universal covering spaces.

In Section 2, we introduce Spanier spaces and based Spanier spaces which are the spaces that their fundamental group is equal to their Spanier groups and based Spanier groups, respectively. By an example we show that these notions are different, although for locally path connected spaces they are the same since for locally path connected spaces we have $\pbs=\psp$ \cite{R}. Also, we show that small generated spaces in the sense of \cite{T2} are based Spanier spaces and hence Spanier spaces but the converse is not true in general.

In Section 3, we prove that for every covering $p:\wt{X}\lo X$ of a space $X$, $p_*\pi_1(\wt{X},\ti{x})$ contains $\pi_1^{sp}(X,x)$ as a subgroup and so $X$ has no simply connected covering space if it is non-homotopically path Hausdorff. Then, we introduce Spanier coverings that are universal coverings in the categorical sense. Also, in this case $\pi_1^{sp}(X,x)=p_*\pi_1(\wt{X},\ti{x})$.

In Section 4, we present the main result of this article which states that a connected and locally path connected space $X$ has a Spanier covering if and only if $X$ is a semi-locally Spanier space, that is, for every $x\in X$ there exists an open neighborhood $U$ such that the homotopy class of every loop in $U$ belongs to $\psp$. Also, we prove that for a connected and locally path connected space $X$ which admits a Spanier covering, for every subgroup $H$ of $\pi_1(X,x)$ which contains $\psp$ there exists a covering $p:\wt{X}_H\lo X$ such that $p_*\pi_1(\wt{X}_H,\ti{x})=H$, where $\ti{x}\in p^{-1}(\{x\})$.

Finally in Section 5, we study the topological properties of the Spanier group in topological fundamental groups. We prove that in connected and locally path connected spaces, $\psp$ is a closed subgroup and hence if $\pi_1^{\tau}(X,x)$ has indiscrete topology, then $X$ is a Spanier space. Using this fact, we show that the topological fundamental group of a connected and locally path connected space $X$ with trivial Spanier group is $T_1$ and also is Hausdorff if $X$ is paracompact.

Throughout this article, all the homotopies between two paths are relative to end points,
$X$ is a path connected space with the base point $x\in X$, and $p:\wt{X}\lo X$ is a path connected covering of $X$ with $\ti{x}\in p^{-1}(\{x\})$ as the base point of $\wt{X}$. Also, by the Spanier group we mean the unbased ones and since in locally path connected spaces based and unbased cases are coincide, we remove prefix unbased.

\section{Spanier spaces}
\begin{definition}
 We call a topological space $X$ the unbased Spanier space if $\pi_1(X,x)=\pi_1^{sp}(X,x)$, and the based Spanier space if $\pi_1(X,x)=\pi_1^{bsp}(X,x)$, for an arbitrary point $x\in X$.
\end{definition}
Since for locally path connected spaces the based and unbased
Spanier groups coincide \cite{R}, we use the terminology the Spanier space but for general spaces we omit the prefix unbased for brevity.
If a space $X$ is not homotopically Hausdorff, then there exists $x\in X$ and a nontrivial loop in $X$ based at $x$ which is homotopic to a loop in every neighborhood $U$ of $x$. Z. Virk \cite{V} called these loops as \emph{small loops} and showed that for every $x\in X$ they form a subgroup of $\pi_1(X,x)$ which is named \emph{small loop group} and denoted by $\pi_1^s(X,x)$. In general, these loops make the \emph{SG (small generated) subgroup}, denoted by $\psg$, which is the subgroup generated by the following set $$\{[\al*\bt*\al^{-1}]\ |\ [\bt]\in\pi_1^s(X,\al(1)),\ \al\in P(X,x)\},$$
where $P(X,x)$ is the set of all paths from $I$ into $X$ with initial point $x$ \cite{V}.
Also, Virk \cite{V} introduced \emph{small loop spaces} which are spaces with all loops as small loops and the authors \cite{T2} introduced \emph{small generated spaces} which their SG subgroup is their fundamental group. The following proposition easily comes from the definitions.
\begin{proposition}
For a topological space X and every $x\in X$, $\pi_1^s(X,x)\leq\psg\leq\pi_1^{bsp}(X,x)\leq\pi_1^{sp}(X,x)$.
\end{proposition}
\begin{example}
Small loop spaces and small generated spaces are based Spanier spaces and therefore by the above proposition they are Spanier spaces.
\end{example}
The following example shows that every Spanier space is not necessarily a based Spanier space.
\begin{example}
Consider the space Y introduced in \cite[Fig. 1]{R} which is obtained by taking the "surface" obtained by rotating the topologist's sine curve about its
limiting arc and then adding a single arc C. This arc C can be easily embedded
into $\mathbb{R}^3$, so as not to intersect the surface portion or the central axis at any other points than its endpoints (see Fig. 1).
\begin{figure}
\center
 \includegraphics[scale=0.3]{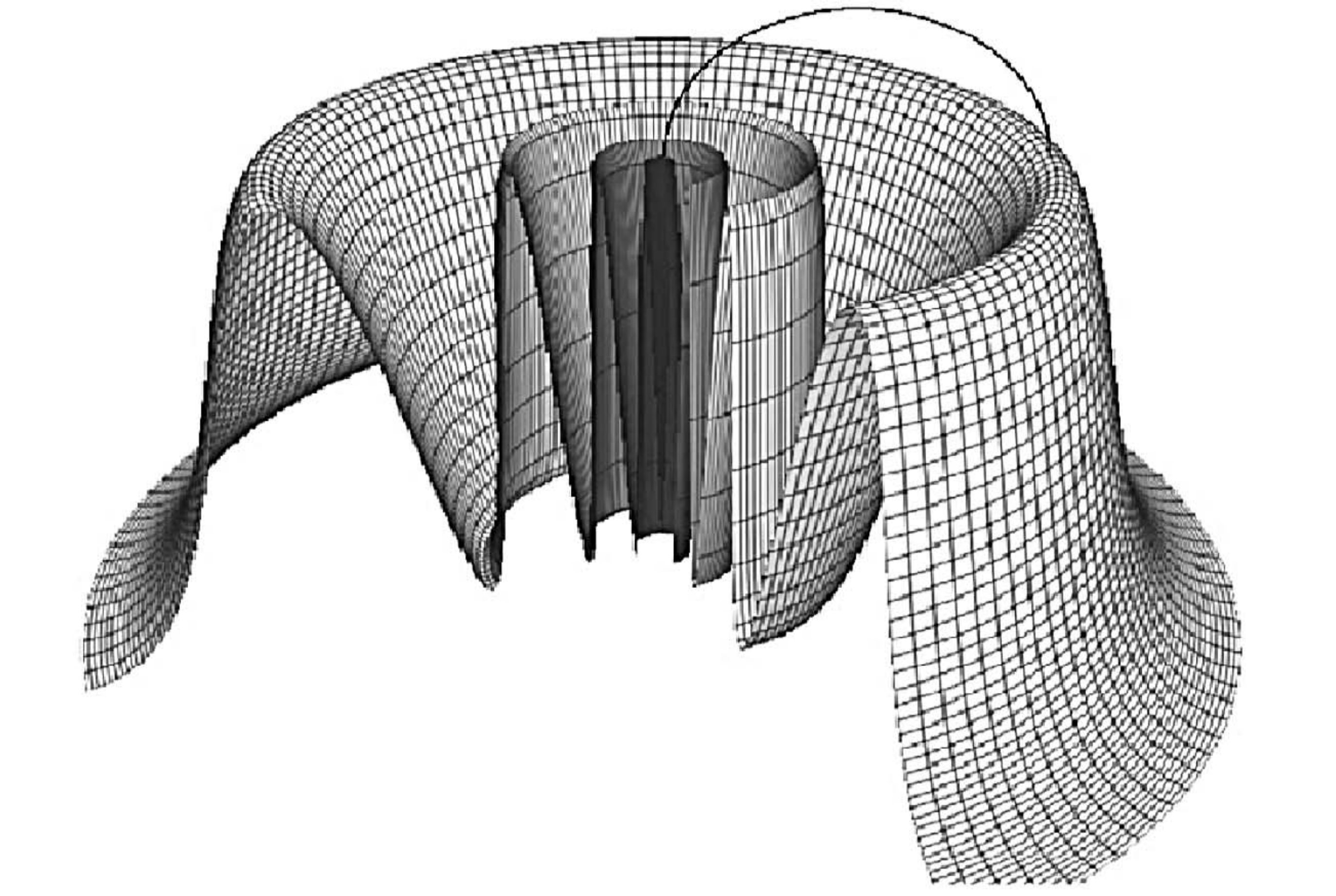}
  \caption{Space $Y$}\label{1}
\end{figure}
By \cite[Proposition 3.1]{R}, the space $Y$ is semi-locally simply connected which implies that $\pi_1^{sg}(Y,x_0)=1$. Fix a point $x_0$ on the surface portion of $Y$. Let $\rho_r$ denote a simple path on the surface starting
at $x_0$, contained in the plane determined by $x_0$ and the central axis, with endpoint at distance r from the central axis.
 Let $\al_r$ be the simple loop with radius $0<r<1$ on the surface. Obviously $\al_r$ is not null homotopic and any
neighborhood of a point of the central axis contains such a loop. For every $0<r<1$ the loops $\rho*\al_r*\rho^{-1}$ (with $\al_r$ appropriately
based) are homotopic to each other and non-trivial and hence $\pi_1^{sp}(Y,x_0)=\pi_1(Y,x_0)$. Note that the space $Y$ is not locally path connected and $\pi_1^{bsp}(Y,x_0)=1$ and hence the equality $\pi_1^{sp}(Y,x_0)=\pi_1^{bsp}(Y,x_0)$ does not hold in general.
\end{example}
\begin{remark}
 Note that since $\pi_1^{sg}(Y)=1$, the space $Y$ is an example of Spanier spaces which are not small generated space.
 \end{remark}
\section{Spanier coverings}
The importance of Spanier groups was pointed out by Conner, Meilstrup, Repovs, Zastrow and Zeljko \cite{CR} and by Fischer, Repoves, Virk and Zastrow \cite{R}. In this section, we study some basic properties of Spanier groups and their relations to the covering spaces.
By convention, the term \emph{universal covering} will always mean a categorical universal object, that is, a covering
$p:\wt{X}\lo X$ with the property that for every covering $q:\wt{Y}\lo X$ with a path connected space $\wt{Y}$ there exists a covering $r:\wt{X}\lo\wt{Y}$ such that $q\circ r= p$. Also, we denote by $\mathcal{COV}$(X) the category of all coverings of $X$ as objects and covering maps between them as morphisms.

\begin{theorem}
For every covering $p:\wt{X}\lo X$ and $x\in X$ the following relations hold: $$\pbs\leq\psp\leq\pst.$$
\end{theorem}
\begin{proof}
Let $\U$ be a cover of $X$ by evenly covered open subsets. It suffices to show that $\pi(\U,x)\leq\pst$. For this, let $[\al]\in\pi(\U,x)$. Then there are open subsets $U_i\in\U$, paths $\al_i$ from $x$ to $x_i\in U_i$ and loops $\bt_i:I\lo U_i$ based at $x_i$, for $i=1,2,...,n$, such that
$$\al\simeq (\al_1*\bt_1*\al_1^{-1})*(\al_2*\bt_2*\al_2^{-1})*...*(\al_n*\bt_n*\al_n^{-1}).$$
Let $\wt{\al_i}$ be the lift of $\al$ with initial point $\ti{x}$, $\wt{\al_i^{-1}}$ be the lift of $\al_i^{-1}$ with initial point $\wt{\al_i}(1)$ and $\wt{\bt_i}=(p|_{V_i})^{-1}\circ\bt_i$ where $V_i$ is the homeomorphic copy of $U_i$ containing $\ti{x_i}=\wt{\al_i}(1)$. Now define $$\wt{\al}=(\wt{\al_1}*\wt{\bt_1}*\wt{\al_1^{-1}})*(\wt{\al_2}*\wt{\bt_2}*\wt{\al_2^{-1}})*...*(\wt{\al_n}*\wt{\bt_n}*\wt{\al_n^{-1}})$$
which is a loop in $\wt{X}$ and $p\circ\wt{\al}\simeq\al$ which implies that $[\al]\in\pst$.
\end{proof}
We know that the image subgroup $\pst$ in $\pi_1(X,x)$
consists of the homotopy classes of loops in $X$ based at $x$ whose lifts to $\wt{X}$ starting
at $\ti{x}$ are loops. Hence the following result holds.
\begin{corollary}
If $\pc$ is a covering and $[\al]\in\psp$, then every lift of $\al$ in $\wt{X}$ is a loop.
\end{corollary}
\begin{corollary}
If the space $X$ is not homotopically path Hausdorff, then $X$ does not admit a simply connected covering space.
\end{corollary}
\begin{corollary}
Let $X$ be a connected, locally path connected and simply connected space. If the action of a group
$G$  on $X$ is properly discontinuous, then $\pi_1^{sp}(X/G)=1$ and therefore $X/G$ is homotopically path Hausdorff.
\end{corollary}
\begin{corollary}
For every covering $p:\wt{X}\lo X$, $\psp$ acts trivially on $p^{-1}(\{x\})$, that is, $\hat{x}.[\al]=\hat{x}$, for all $\hat{x}\in p^{-1}(\{x\})$ and $[\al]\in \psp$.
\end{corollary}

\begin{theorem}
Every covering space of a (based or unbased) Spanier space $X$ is homeomorphic to $X$.
\end{theorem}
\begin{proof}
Let $p:\wt{X}\lo X$ be a covering of a Spanier space $X$. Then by Theorem 3.1 $\psp\leq p_*\pi_1(\wt{X},\tilde{x})\leq\pi_1(X,x)=\psp$, for each $x\in X$ which implies that $p:\wt{X}\lo X$ is a one sheeted covering of $X$. Hence  $\wt{X}$ is homeomorphic to $X$.
\end{proof}
\begin{definition}
By a $\tx{(based) Spanier covering}$ of a topological space $X$ we mean a covering $p:\wt{X}\lo X$ such that $\wt{X}$ is a (based) Spanier space.
\end{definition}
The following proposition comes from definitions and Theorem 3.1.
\begin{proposition}
For a space $X$ the following statements hold.\\
(i) Every based Spanier covering is a Spanier covering.\\
(ii) If $X$ has a based Spanier covering, then $\pbs=\psp$.
\end{proposition}
\begin{lemma}
If $\pc$ is a covering and $[\al]$ belongs to $\pi_1^{bsp}(\wt{X},\ti{x})$ (or $\pi_1^{sp}(\wt{X},\ti{x}))$, then $[p\circ\al]$ belongs to $\pbs$ (or $\psp)$.
\end{lemma}
\begin{proof}
Let $\U$ be a cover of $X$ by evenly covered open neighborhoods and then $\V=p^{-1}(\U)$ is an open cover of $\wt{X}$. Since $\pi_1^{bsp}(\wt{X},\ti{x})\sub\pi(\V,\ti{x})$, $[\al]\in\pi(\U,\ti{x})$ and hence there are pointed open subsets $(V_1,y_1),(V_2,y_2),...,(V_n,y_n)\in\V$, the paths $\al_i$ initiated from $\ti{x}$ and loops $\bt_i:I\lo V_i$ based at $\al_i(1)=y_i$, for $i=1,2,...,n$, such that $$\al\simeq (\al_1*\bt_1*\al_1^{-1})*(\al_2*\bt_2*\al_2^{-1})*...*(\al_n*\bt_n*\al_n^{-1}).$$
If $\la_i=p\circ\al_i$ and $\theta_i=p\circ\bt_i$, for $i=1,2,...,n$, then the $\theta_i$'s are loops in $(U_i,x_i)=p((V_i,y_i))\in\U$ based at $\la_i(1)=x_i$. Therefore $$p\circ\al\simeq (\la_1*\theta_1*\la_1^{-1})*(\la_2*\theta_2*\la_2^{-1})*...*(\la_n*\theta_n*\la_n^{-1})$$
which implies that $[p\circ\al]\in\pi(\U,x)$. Since every open cover of $X$ has a refinement by evenly covered open subset, $[p\circ\al]\in\pbs$.
\end{proof}
\begin{theorem}
(i) A covering $p:\wt{X}\lo X$ is a based Spanier covering if and only if $\pi_1^{bsp}(X,x)=p_*\pi_1(\wt{X},\ti{x})$.\\
(ii) A covering $p:\wt{X}\lo X$ is a Spanier covering if and only if $\pi_1^{sp}(X,x)=p_*\pi_1(\wt{X},\ti{x})$.
\end{theorem}
\begin{proof}
(i) By definition of based Spanier coverings, $\pi_1(\wt{X},\ti{x})=\pi_1^{bsp}(\wt{X},\ti{x})$. Using Lemma 3.9 and Proposition 2.2 we have $\pi_1^{bsp}(X,x)=p_*\pi_1(\wt{X},\ti{x})$. Conversely, let $[\al]\in\pi_1(\wt{X},\ti{x})$ and $\V$ be an open cover of $\wt{X}$ such that $\U=p(\V)$ be an open cover of $X$ by evenly covered open subsets of $X$. Since $p_*\pi_1(\wt{X},\ti{x})=\pi_1^{bsp}(X,x)$, $[p\circ\al]\in\pi_1^{bsp}(X,x)$ and hence there are pointed open subsets $(U_1,x_1),(U_2,x_2),...,(U_n,x_n)\in\U$, the paths $\al_i$ initiated from $x$ and loops $\bt_i:I\lo U_i$ based at $\al_i(1)=x_i$, for $i=1,2,...,n$, such that $$p\circ\al\simeq (\al_1*\bt_1*\al_1^{-1})*(\al_2*\bt_2*\al_2^{-1})*...*(\al_n*\bt_n*\al_n^{-1}).$$ Let $\wt{\al_i}$ be the lift of $\al_i$ with initial point $\ti{x}$ and $\wt{\bt_i}=(p|_{V_i})^{-1}\circ\bt_i$ be the loop with base point $\wt{\al_i}(1)$, where $V_i$ is the homeomorphic copy of $U_i$ in $\wt{X}$ which contains $\wt{\al_i}(1)$, then  $$[(\wt{\al_1}*\wt{\bt_1}*\wt{\al_1}^{-1})*...*(\wt{\al_n}*\wt{\bt_n}*\wt{\al_n}^{-1})]\in\pi(\V,\ti{x}).$$
If $\wt{\al_{i}^{-1}}$ is the lift of $\al_{i}^{-1}$ with initial point $\wt{\bt_i}(1)$, then $\wt{\al_{i}^{-1}}=\wt{\al_i}^{-1}$
 and hence
 Since $p_*$ is injective we have $$[\al]=[(\wt{\al_1}*\wt{\bt_1}*\wt{\al_1}^{-1})*...*(\wt{\al_n}*\wt{\bt_n}*\wt{\al_n}^{-1})]\in\pi(\V,\ti{x}).$$
(ii) By a similar proof to (i) the result holds.
\end{proof}
\begin{lemma}
Let $X$ have a based Spanier covering $p:\wt{X}\lo X$ and $\al$ be a loop in $X$ which has a closed lift $\wt{\al}$ in $\wt{X}$, then $[\al]\in\pbs=\psp$. Consequently, if a loop $\al$ in $X$ has a closed lift $\wt{\al}$ in $\wt{X}$, then every lift of $\al$ in $\wt{X}$ is also closed.
\end{lemma}
\begin{proof}
 Assume that a lift $\wt{\al}$ of $\al$ is closed and $\U$ is a cover of $X$ by evenly covered open neighborhood pairs and let $\V=p^{-1}(\U)$. Since $[\wt{\al}]\in\pi_1(\wt{X},\ti{x})=\pi_1^{bsp}(\wt{X},\ti{x})\leq\pi(\V,\ti{x})$, there are $(V_1,y_1),(V_2,y_2),...,(V_n,y_n)\in\V$, the paths $\al_i$ from $\ti{x}$ to $y_i$ and loops $\bt_i:I\lo V_i$ based at $\al_i(1)=y_i$, for $i=1,2,...,n$, such that $$\wt{\al}\simeq (\al_1*\bt_1*\al_1^{-1})*(\al_2*\bt_2*\al_2^{-1})*...*(\al_n*\bt_n*\al_n^{-1}).$$
 If $(U_i,x_i)=p((V_i,y_i))$, then $(U_i,x_i)\in\U$ since $\V=p^{-1}(\U)$ and hence $$\al\simeq p\circ\left((\al_1*\bt_1*\al_1^{-1})*(\al_2*\bt_2*\al_2^{-1})*...*(\al_n*\bt_n*\al_n^{-1})\right)\in\pi(\U,x),$$
 which implies that $[\al]\in\pbs=\psp$ since $\U$ is arbitrary.
\end{proof}
\begin{theorem}
A Spanier covering of a locally path connected space $X$ is the universal object in the category $\mathcal{COV}$(X).
\end{theorem}
\begin{proof}
Assume that $p:\wt{X}\lo X$ is a Spanier covering of $X$ and $q:\wt{Y}\lo X$ is another covering. By Proposition 2.2 and Theorem 3.10 $p_*\pi_1(\wt{X})=\pi_1^{sp}(X)\leq q_*\pi_1(\wt{Y})$ which implies that there exists $f:\wt{X}\lo\wt{Y}$ such that $q\circ f=p$.
\end{proof}
\section{Existence}
Suppose $\pc$ is a Spanier covering. Every point $x\in X$ has a neighborhood $U$ having a
lift $\wt{U}\sub\wt{X}$ projecting homeomorphically to $U$ by $p$. Each loop $\al$ in $U$ lifts to a loop $\wt{\al}$
in $\wt{U}$, and $[\wt{\al}]\in\pi_1^{sp}(\wt{X},\ti{x})$ since $\pi_1(\wt{X},\ti{x})=\pi_1^{sp}(\wt{X},\ti{x})$. So, composing this with Lemma 3.9, $[\al]\in\psp$. Thus the space $X$ has the following
property: \emph{Every point $x\in X$ has a neighborhood $U$ such that
$i_*\pi_1(U,x)\leq\pi_1^{sp}(X,x)$.}
\begin{definition}
 We call a space $X$ semi-locally Spanier space if and only if for each $x\in X$ there exists an open neighborhood $U$ of $x$ such that $i_*\pi_1(U,x)\leq\pi_1^{sp}(X,x)$, where $i:U\lo X$ is the inclusion map.
\end{definition}
\begin{example}
Every Spanier space is a semi-locally Spanier space. Also, the product $X\times Y$ is a semi-locally Spanier space if $X$ is a Spanier space and $Y$ is either locally simply connected or locally path connected and semi-locally simply connected.
 If $(X,x)$ is a pointed Spanier space and $(Y,y)$ is a locally nice space like the above, then one point union $X\vee Y$ is a semi-locally Spanier space.
\end{example}
The following lemma easily comes from definitions.
\begin{lemma}
Let $\U$ be an open cover of a space $X$ and $\theta$ be a path in $X$ with $\theta(0)=x_1$ and $\theta(1)=x_2$. Then $\pi(\U,x_2)=\vf_{\theta}\pi(\U,x_1)$, where $\vf_{\theta}:\pi_1(X,x_1)\lo\pi_1(X,x_2)$ is the isomorphism given by $\vf_{\theta}([\al])=[\theta^{-1}*\al*\theta]$.
\end{lemma}
\begin{corollary}
For a space $X$ and $x_1,x_2\in X$, the homomorphism $\phi_{\theta}:{\pi_1^{sp}(X,x_1)}\lo {\pi_1^{sp}(X,x_2)}$ defined as the same as $\vf_{\theta}$ is an isomorphism.
\end{corollary}
\begin{theorem}
A connected and locally path connected space $X$ has Spanier covering if and only if $X$ is a semi-locally Spanier space.
\end{theorem}
\begin{proof}
The discussion at the beginning of the section follows necessity. For sufficiently, we show that there exists an open cover $\U$ of $X$ such that $\pi(\U,x)=\psp$ which implies the existence of a covering $\pc$ such that $\pst=\psp$, which is a Spanier covering by Theorem 3.10.

Let $\U$ be the open cover of $X$ consisting of the path connected open subsets of $X$ such that $i_*\pi_1(U,y)\leq\pi_1^{sp}(X,y)$, for $U\in\U$, where $i:U\lo X$ is the inclusion map. By definition of the Spanier group we have $\psp\leq\pi(\U,x)$. Also, by the choice of $\U$ and Corollary 4.5 we have $\pi(\U,x)\leq\psp$. Therefore $\pi(\U,x)=\psp$, as desired.
\end{proof}
Using Theorem 1.1 we have the following result.
\begin{theorem}
Suppose $X$ is a connected, locally path connected and semi-locally
 Spanier space. Then for every subgroup $H\leq\pi_1(X,x)$ containing $\pi_1^{sp}(X,x)$, there exists a covering $p:\wt{X}_H\lo X$ such that $p_*\pi_1(\wt{X}_H,\ti{x})=H$, for a suitably chosen base point
$\ti{x}\in \wt{X}_H$.
\end{theorem}
The following theorem gives a sufficient condition for a fibration with unique path lifting property to be a covering.
\begin{theorem}
Let $\pc$ be a fibration with unique path lifting property, where $X$ is connected and locally path connected. Then $\pc$ is a covering if $X$ is a semi-locally Spanier space.
\end{theorem}
\begin{proof}
By \cite[Theorem 2.5.12]{S} $\pc$ is a covering if and only if there exists an open cover $\U$ of $X$ such that $\pi(\U,x)\sub p_*\pi_1(\wt{X},\ti{x})$. Hence it suffices to let $\U$ as be chosen as in the proof of Theorem 4.5.
\end{proof}
\section{The topology of Spanier subgroups}
In this section, for a space $X$ and $x\in X$, by $\pi_1^{qtop}(X,x)$ we mean the notorious topological fundamental group endowed with the quotient topology inherited from loop space under natural map $\Omega(X,x)\lo\pi_1(X,x)$ that make it a \emph{quasi-topological group}. A \emph{quasi-topological group} $G$ is a group with a topology such that inversion $g\lo g^{-1}$ and all translations are continuous. For more details, see \cite{B, Br, Cal}. Also, $\pi_1^{\tau}(X,x)$ is the fundamental group endowed with another topology introduced by Brazas \cite{Br2}. In fact, the functor $\pi_1^{\tau}$ removes
the smallest number of open sets from the topology of $\pt$ so that make it a topological group. Here, by topological fundamental group we mean $\pi_1^{\tau}(X,x)$.

By \cite[Theorem 5.5]{B}, the connected coverings of a connected and locally path connected space $X$ are classified by conjugacy classes of open subgroups of $\pi_1^{qtop}(X,x)$ and since $\pi_1^{\tau}(X,x)$ and $\pi_1^{qtop}(X,x)$ have the same open subgroups \cite[Corollary 3.9]{Br2} we have the same result for open subgroups of $\pi_1^{\tau}(X,x)$. Using this fact and Theorem 1.1, for every open cover $\U$ of $X$, $\pi(\U,x)$ is an open subgroup of $\pi_1^{\tau}(X,x)$ and since $\pi_1^{\tau}(X,x)$ is a topological group, $\pi(\U,x)$ is a closed subgroup which implies that $\psp$ is a closed subgroup of $\pi_1^{\tau}(X,x)$. Hence we have the following proposition.
\begin{proposition}
For a connected and locally path connected space $X$, $\psp$ is a closed subgroup of $\pi_1^{\tau}(X,x)$, for every $x\in X$.
\end{proposition}
 Using the above proposition the Spanier group of connected and locally path connected spaces contains the closure of the trivial element of the topological fundamental group. Hence we have the following corollary.
\begin{corollary}
Let $X$ be a connected and locally path connected space and $x\in X$. If $\pi_1^{\tau}(X,x)$ has indiscrete topology, then $X$ is a Spanier space.
\end{corollary}
\begin{corollary}
Let $X$ be a connected and locally path connected space and $x\in X$. If $\psp=1$, then $\pi_1^{\tau}(X,x)$ has $T_1$ topology.
\end{corollary}
In the following example we show that the locally path connectedness is a necessary condition for closeness of $\psp$.
\begin{example}
Let the space $Z\sub \mathbb{R}^3$ be consist of a rotated topologists' sine curve (as shown in Figure.2), the "outer cylinder"
at radius 1, where this surfaces tends to, and horizontal segment $\la$ from $(0,0,0)$ to $(1,0,0)$ that attaches them. The segment $\la$ intersects the inner portion at points $(\fr{1}{n},0,0)$. Denote by $\la_n:I\lo Z$ the line segment from $(0,0,0)$ to $(\fr{1}{n},0,0)$, $\al_n:I\lo Z$ the simple loop with radius $\fr{1}{n}$, for $n\in\N$, and $\al$ the simple loop with radius 1. Obviously, $\la_n*\al_n*\la_n^{-1}\rightarrow\la*\al*\la^{-1}$ in uniform topology which is equivalent to compact open topology in metric space $Z$. Since $[\la_n*\al_n*\la_n^{-1}]= 1$ for all $n\in\N$, $[\la*\al*\la^{-1}]\in \ov{\{1\}}$, but $[\la*\al*\la^{-1}]\neq1$. Also, if $\U$ is an open cover of $Z$ such that for every $U\in\U$, $diam(U)<1$, then $\pi(\U,0)=1$ which implies that $\pi_1^{sp}(Z,0)=1$.
\end{example}
\begin{figure}
\center
 \includegraphics[scale=0.3]{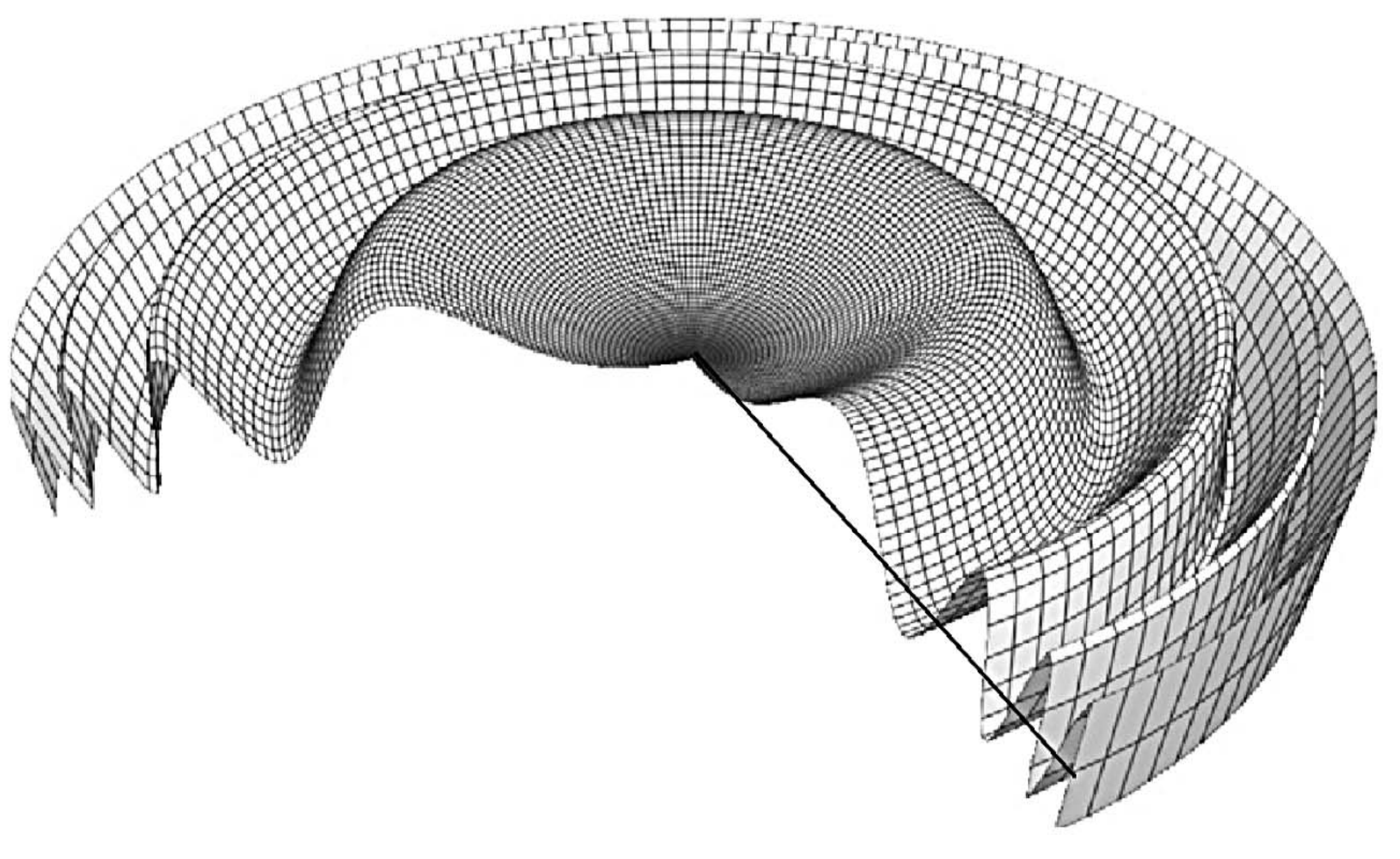}
  \caption{Space $Z$}\label{1}
\end{figure}
The authors \cite{T1} proved that for a first countable, simply connected and locally path connected space $X$, if $A\sub X$ is a closed path connected subset, then the quotient space $X/A$ has indiscrete topological fundamental group. Therefore by Corollary 5.2 we have the following theorem that gives a family of Spanier spaces.
\begin{theorem}
For a first countable, simply connected and locally path connected space $X$, if $A\sub X$ is a closed path connected subset, then the quotient space $X/A$ is a Spanier space.
\end{theorem}
\begin{remark}
Note that by assumptions of the above theorem, the quotient space $X/A$ is locally path connected and so based and unbased Spanier groups coincide.
\end{remark}
We recall that a space $X$ is called shape injective, if the natural homomorphism $\vf:\pi_1(X,x)\lo\check{\pi}_1(X,x)$ is injective, where $\check{\pi}_1(X,x)$ is the first shape group of $(X,x)$ (see \cite[Section 3]{FZ} for further details).
In \cite{Br2}, it is proved that topological fundamental groups of shape injective spaces are Hausdorff. Since the spaces with trivial Spanier group are not necessarily shape injective (see \cite[Step 18, Section 6]{R}), triviality of Spanier groups can not certify the Hausdorffness of topological fundamental groups in general. In the following, we provide conditions that guarantee this.

We recall that an open cover $\U$ of $X$ is normal if it admits a partition of unity subordinated to $\U$. Also, every open cover of a paracompact space is normal (see \cite{MS}).
\begin{definition}(\cite{Lub})
A space $X$ is small loop homotopically Hausdorff if for each $x\in X$
and each loop $\al$ based at $x$, if for each normal open cover $\U$ of $X$, $[\al]\in\pi(\U, x)$, then $[\al]=1$.
\end{definition}
\begin{proposition}
Suppose $X$ is a connected, locally path connected and paracompact space. Then $\psp=1$ if and only if $X$ is shape injective.
\end{proposition}
\begin{proof}
Since every open cover of a paracompact space is normal, $\psp=1$ is equivalent to small loop homotopically Hausdorffness of $X$. Hence the result follows from
 Proposition 2.4 of \cite{Lub} which states that shape injectivity and small loop homotopically Hausdorffness are equivalent.
\end{proof}
\begin{corollary}(A criteria for Hausdorffness of $\pi_1^{\tau}(X,x)$)\\
Let $X$ be a connected, locally path connected and paracompact space. If $\psp=1$, then $\pi_1^{\tau}(X,x)$ is Hausdorff.
\end{corollary}
Note that since $\pi_1^{\tau}(X,x)$ is a topological group, by assumption of the above corollary $\pi_1^{\tau}(X,x)$ is regular.
Since $\pi_1^{sp}(HE)=1$ (Proposition 3.6 of \cite{R}) and metrizable spaces are paracompact, the topological fundamental group of the Hawaiian earing, $\pi_1^{\tau}(HE,0)$, is Hausdorff.
Moreover, since every planar space $X$ is shape injective \cite[Theorem 2]{FZ}, $\pi_1^{\tau}(X,x)$ is Hausdorff.
\begin{remark}
Note that since the topology of $\pi_1^{qtop}(X,x)$ is finer than the topology of $\pi_1^{\tau}(X,x)$, all the above results hold for $\pi_1^{qtop}(X,x)$.
\end{remark}
\subsection*{Acknowledgements}
The authors would like to thank the referee for the valuable comments that help improve the manuscript.












\end{document}